\documentclass[a4paper,10pt]{article}
\usepackage{amsmath,amssymb,amsfonts,amsthm}
\usepackage{graphicx,color,subfigure,enumerate,multirow}
\usepackage[]{hyperref}
\usepackage{fullpage}

\newtheorem{thm}{Theorem}
\newtheorem{prop}[thm]{Proposition}
\newtheorem{cor}[thm]{Corollary}
\newtheorem{lem}[thm]{Lemma}

\newtheorem{rmk}[thm]{Remark}

\title{Courant-sharp eigenvalues of a two-dimensional torus}
\author{Corentin L\'ena\footnote{Department of Mathematics \emph{Guiseppe Peano}, University of Torino, Via Carlo Alberto, 10, 10123 Turin, Italy 
\texttt{clena@unito.it}}}

\begin{document}
\maketitle
\begin{abstract}
In this paper, we determine, in the case of the Laplacian on the flat two-dimensional torus $(\mathbb{R}/\mathbb{Z})^2\,$, all the eigenvalues 
having an eigenfunction 
which satisfies Courant's theorem 
with equality (Courant-sharp situation). Following the strategy of {\AA}.~Pleijel (1956), the proof is a combination of  a lower bound (\`a la Weyl) of the counting function, with an explicit remainder term, and of a Faber--Krahn inequality for 
domains on the torus (deduced as in B\'erard-Meyer from an isoperimetric inequality), with an explicit upper bound on the area.
\end{abstract}



\section{Introduction}

Let us first recall two classical results on the eigenvalues and eigenfunctions of the Dirichlet-Laplacian on a bounded domain $\Omega$ in the plane. According to a well-known result by R.~Courant (in \cite{Cou}), an eigenfunction associated with the $k$-th eigenvalue $\lambda_k(\Omega)$ of this operator has at most $k$ nodal domains. In \cite{Pl}, {\AA}.~Pleijel sharpened this result by showing that, for a given domain, an eigenfunction associated with $\lambda_k(\Omega)$ has less than $k$ nodal domains, except for a finite number of indices $k\,$. This  was generalized in \cite{BeMe} by P.~B\'erard and D.~Meyer to the case of a compact Riemannian manifold, with or without boundary, in any dimension. It has been shown by I. Polterovich in \cite{Po}, using estimates from \cite{ToZe}, that the analogous result also holds for the Neumann-Laplacian on a planar domain with a piecewise-analytic boundary.

These results leave open the question of determining, for a given domain or manifold, the cases of equality.  It is stated in \cite{Pl} that when $\Omega$ is a square, equality can only occurs for eigenfunctions having one, two or four nodal domains, associated with the first, the second (which has multiplicity two), and the fourth eigenvalue respectively. The proof in \cite{Pl} is however incomplete and was corrected by P. B\'erard and B. Helffer in   \cite{BeHe}. The cases of an equilateral torus and  an equilateral triangle are investigated in \cite{BeHe15}, and the case of the Neumann-Laplacian on a square is treated in  \cite{HePe}. In this note, we will show that for the flat torus
$(\mathbb{R}/\mathbb{Z})^2\,$, equality holds only for eigenfunctions having one or two nodal domains, respectively associated with the first and the second eigenvalue (this last eigenvalue has multiplicity four). This complements the result \cite[Theorem 7.1]{HH}, which determines the cases of equality for a flat torus $(\mathbb{R}/\mathbb{Z})\times(\mathbb{R}/b\mathbb{Z})$
with $b^2$ irrational. 

Let us give a more precise statement of the above result, and fix some notation that will be used in the following. In the rest of this paper, $\mathbb{T}^2$ stands for the two-dimensional torus
\[\mathbb{T}^2=(\mathbb{R}/\mathbb{Z})^2\]
equipped with the standard flat metric, and $-\Delta_{\mathbb{T}^2}$ stands for the Laplace-Beltrami operator on $\mathbb{T}^2\,$. If $\Omega$ is an open set in $\mathbb{T}^2$ with a piecewise-$C^1$ boundary, we write $(\lambda_k(\Omega))_{k\ge 1}$ for the eigenvalues of $-\Delta_{\mathbb{T}^2}$ on $\Omega$ with a Dirichlet boundary condition on $\partial \Omega\,$, arranged in non-decreasing order and counted with multiplicity. In particular, $\lambda_k(\mathbb{T}^2)$ is the $k$-th eigenvalue of $-\Delta_{\mathbb{T}^2}$ (in that case the boundary is empty). If $u$ is an eigenfunction of $-\Delta_{\mathbb{T}^2}\,$, we call \emph{nodal domains of $u$} the connected components of $\mathbb{T}^2\setminus u^{-1}(\{0\})\,$, and we denote by $\mu(u)$  the cardinal of the set of nodal domains.
With any eigenvalue $\lambda$ of $-\Delta_{\mathbb{T}^2}\,$, we associate the integer
\[\nu(\lambda)=\min\{k\in \mathbb{N}^{*} \,:\,\lambda_k(\mathbb{T}^2)=\lambda \}\,.\]
Following \cite{HHOT}, we say that an eigenvalue $\lambda$ of $-\Delta_{\mathbb{T}^2}$ is a \emph{Courant-sharp eigenvalue of $\mathbb{T}^2$} if there exists an associated eigenfunction $u$ with $\mu(u)=\nu(\lambda)\,$. We will prove the following result.

\begin{thm} \label{thmPlejielTorus} The only Courant-sharp eigenvalues of $\mathbb{T}^2$ are $\lambda_k(\mathbb{T}^2)$ with $k\in\{1,2,3,4,5\}\,$.
\end{thm}

The proof follows the approach used  by \AA. Pleijel in \cite{Pl} and in the case of a compact manifold by B\'erard-Meyer in \cite{BeMe} (see also \cite{Be1,Be2}). In Section \ref{secIso}, we establish an isoperimetric inequality and we use it to prove a Faber-Krahn type inequality for domains in $\mathbb{T}^2\,$. The inequality that we use is a special case of the more general result in \cite[7]{HoHuMo} but, for the sake of completeness, we give an alternative proof that avoids the use of variational arguments. In Section \ref{secProof}, we combine this information with Weyl's law to show that large eigenvalues cannot be Courant-sharp. In the very simple setting under consideration, we can get explicit version of all these estimates, and we are therefore able to prove Theorem \ref{thmPlejielTorus}. 

\paragraph{Acknowledgements} I want first and foremost to thank Bernard Helffer for introducing me to this problem and for numerous discussions and advices. I also thank Pierre B\'erard for explaining to me the approach used in \cite{BeMe}, pointing out the reference \cite{HoHuMo}, and suggesting several improvements. I thank Virginie Bonnaillie-No\"el for her corrections and  Thibault Paolantini for his help with the topological questions in Section \ref{secIso}. This work was partially supported by the ANR (Agence Nationale de la Recherche), project OPTIFORM n$^\circ$
ANR-12-BS01-0007-02, and by the ERC, project COMPAT n$^\circ$ ERC-2013-ADG.

\section{Two inequalities}
\label{secIso}

\subsection{Isoperimetric inequality}

The first tool in the proof of  Theorem \ref{thmPlejielTorus} is the isoperimetric inequality. We consider open sets $\Omega$ in $\mathbb{T}^2$ that satisfy the following properties:
\begin{enumerate}[(i)]
	\item $\partial \Omega$ is piecewise-$C^1$;
	\item $\Omega$ is \emph{without crack}, that is to say the interior of $\overline{\Omega}$ is equal to $\Omega\,$.
\end{enumerate} 
According to classical regularity results for eigenfunctions of the Laplace-Beltrami operator on surfaces, nodal domains satisfy both properties, and furthermore they are connected by definition. However, since we will apply the isoperimetric inequality to the level sets of eigenfunctions, we need to establish the result without the connectedness assumption. The regularity of the level sets is discussed at the end of Subsection \ref{subsecFaberKrahn}.

If $\Omega$ is an open set in $\mathbb{T}^2$ satisfying (i) and (ii), we write $A(\Omega)$ for its area and $\ell(\partial \Omega)$ for the length of $\partial \Omega\,$. 
The following statement is the first case of \cite[7]{HoHuMo}. For the sake of completeness, we give a proof which differs from the one in \cite{HoHuMo}. Let us however note that the result in \cite{HoHuMo} is more general: all values of $A(\Omega)$ are treated, and no assumption on the regularity of $\partial \Omega$ is made.  

\begin{prop} 
\label{propIsoIneqTorus}
If $A(\Omega)\le \frac{1}{\pi}\,$, we have the inequality
\begin{equation}
	\label{eqIsoIneqTorus}
	\ell(\partial\Omega)^2\ge 4\pi A(\Omega)\,.
\end{equation}
\end{prop}

Let us first remark that without loss of generality, we can assume that $\Omega$ is connected. Indeed, let us assume that Inequality \eqref{eqIsoIneqTorus} has been proved for connected sets, and let $\Omega$ be a general open set satisfying (i) and (ii), with $A(\Omega)\le \frac{1}{\pi}\,$. We can write 
\[\Omega=\bigcup_{j=1}^K\Omega_j\,,\]
where $K$ is some integer and the $\Omega_j$'s are the connected components of $\Omega\,$. The $\Omega_j$'s satisfy (i) and (ii), and since
\[A(\Omega)=\sum_{j=1}^K A(\Omega_j)\,,\]
 $A(\Omega_j)\le \frac{1}{\pi}$ for each $j\,$. By hypothesis,
we have
\[\ell(\partial \Omega_j)^2\ge 4\pi A(\Omega_j)\]
for each $j\,$. From assumption (ii), we have
\[\ell(\partial \Omega)^2=\left(\sum_{j=1 }^K\ell(\partial \Omega_j)\right)^2\ge \sum_{j=1}^K\ell(\partial \Omega_j)^2\,.\]
Summing all the inequalities, we obtain \eqref{eqIsoIneqTorus}.

In the rest of the proof, we assume that $\Omega$ is connected. Let us then remark that $\partial \Omega$ is the disjoint union of a finite number of simple curves embedded in $\mathbb{T}^2$:
\begin{equation}
	\label{eqBoundary}
	\partial \Omega=\bigcup_{i=1}^N C_i\,.
\end{equation}

To establish Inequality \eqref{eqIsoIneqTorus}, we distinguish between the following cases, where the term \emph{trivial} stands for \emph{zero-homologous}:
\begin{enumerate}
	\item every curve $C_i$ is trivial;
	\item at least one curve $C_{i_0}$ is  non-trivial.
\end{enumerate}

\subsubsection{All the curves are trivial}

We distinguish between two sub-cases:
\begin{enumerate}
	\item $\Omega$ is homeomorphic to a disk;
	\item $\Omega$ is not homeomorphic to a disk. 
\end{enumerate}
In sub-case 1, $N=1\,$. In sub-case 2,  $\mathbb{T}^2\setminus \Omega$ is the disjoint union of a finite number of regions that are homeomorphic to a disk.

\paragraph{Sub-case 1: $\Omega$ is homeomorphic to a disk} ~\\

We consider the covering of $\mathbb{T}^2$ by the plane $\mathbb{R}^2$ defined as 
\[\begin{array}{cccc}
			\Pi:&\mathbb{R}^2&\to & \mathbb{T}^2\\
			    &(x,y)       &\mapsto & (x \mbox{ mod }1, y \mbox{ mod } 1)\\ 
  \end{array}\]
  (this is the universal covering).
  The pullback $\Pi^{-1}(\Omega)$ is  the union of an infinite number of connected components, each one being homeomorphic to a disk and having the same area and perimeter as $\Omega\,$. Applying the classical isoperimetric inequality in the plane to one of these components, we obtain \eqref{eqIsoIneqTorus}.
  
  \begin{rmk} Let us note that the inequality holds in this case whether or not a disk of area $A(\Omega)$ can be contained in the torus $\mathbb{T}^2$. 
  \end{rmk} 
  
  \paragraph{Sub-case 2: $\Omega$ is not homeomorphic to a disk} ~\\
  
  We write 
  \[\mathbb{T}^2\setminus \Omega =\bigcup_{i=1}^N D_i\,,\]
  where $N$ is the same integer as in Equation \eqref{eqBoundary} and $D_i$ is homeomorphic to a disk, with $\partial D_i= C_i\,$.\\
  We have 
  \[\sum_{i=1}^NA(D_i)=1-A(\Omega)\,.\]
  and 
  \[\ell(\partial\Omega)=\sum_{i=1}^N\ell(C_i)\,.\]
  We apply the former case of the isoperimetric inequality to each $D_i\,$, and obtain
  \[\ell(C_i)^2\ge 4\pi A(D_i)\, .\]
  Summing all the inequalities, we get 
  \[\ell(\partial \Omega)^2=\left(\sum_{i=1}^N\ell(C_i)\right)^2\ge \sum_{i=1}^N\ell(C_i)^2\ge 4\pi\sum_{i=1}^N A(D_i)=4\pi(1-A(\Omega))\,.\]
  If $1-A(\Omega)\ge 
  A(\Omega)\,$, that is to say if 
  $A(\Omega)\le \frac{1}{2}\,$, we obtain \eqref{eqIsoIneqTorus}.
  
  \subsubsection{At least one curve is non-trivial}
  
  Let us first point out that in this case, there are actually two non-trivial curves. We give one possible proof of this topological fact, using Stokes's formula. 
  
  Let us first introduce some notation. We use $x$ and $y$ to denote the coordinate functions $(x,y)\mapsto x$ and $(x,y)\mapsto y$ on $\mathbb{R}^2$. We will use the two $1$-forms on $\mathbb{T}^2$ obtained as the push-forwards of $dx$ and $dy$ by the mapping $\Pi\,$, which we will also write $dx$ and $dy\,$. If $C$ is a piecewise-$C^1$, closed, simple, and oriented curve on $\mathbb{T}^2\,$, both numbers
  \begin{equation*}
  	p=\int_{C}dx \mbox{ and } q=\int_{C}dy
  \end{equation*}
  are integers, characteristic of the homology class of $C$. In particular, $C$ is zero-homologous if, and only if, $p=q=0\,$. 
  
  According to Stokes's formula,
  \begin{equation*}
  	\sum_{i=1}^N\int_{C_i}dx=\int_{\partial\Omega}dx=\int_{\Omega}d(dx)=0\,,
  \end{equation*}
  and similarly for $dy\,$,
  \begin{equation*}
  	\sum_{i=1}^N\int_{C_i}dy=\int_{\partial\Omega}dy=\int_{\Omega}d(dy)=0\,.
  \end{equation*}
  
  If one of the curves in the boundary is non-trivial, at least one term in one of the sums is non-zero, and thus some other term in the same sum must be non-zero, meaning that some other curve must be non-trivial.
  
  We remark finally that a non-trivial curve has length at least $1$. Since $\partial \Omega$ contains at least two non-trivial curves, $\ell(\partial \Omega)\ge 2\,$. If $A(\Omega)\le \frac{1}{\pi}\,$, this implies
  \[\ell(\partial \Omega)^2\ge 4\pi A(\Omega)\,.\]
  
  Since $\frac{1}{\pi}<\frac{1}{2}\,$, all the cases taken together yield Proposition \ref{propIsoIneqTorus}.

  \subsection{Faber-Krahn inequality}
  \label{subsecFaberKrahn}
  
  As shown for instance by P. B\'erard and D. Meyer  in \cite{BeMe}, an isoperimetric inequality can be used to obtain a Faber-Krahn inequality. The form of this inequality is here particularly simple.
  
  \begin{prop} 
  \label{propFaberKrahn}
 If $\Omega$ is an open and connected set in $\mathbb{T}^2$, without crack and with a piecewise $C^1$-boundary $\partial \Omega$ such that $A(\Omega)\le \frac{1}{\pi}\,$, then 
 \begin{equation}
 \label{eqIneqFaberKrahn}
 \lambda_1(\Omega)A(\Omega)\ge \pi j_{0,1}^2.
 \end{equation}
\end{prop}

The constant $j_{0,1}$  in Equation \eqref{eqIneqFaberKrahn} is the first positive zero of the Bessel function of the first type $J_0\,$. Let us note that $\pi j_{0,1}^2$ is the value of the product $\lambda_1(D)|D|\,$, where $D$ is a disk in the plane $\mathbb{R}^2$ (any disk, since the expression is invariant by scaling).

As in the planar case, the proof uses Schwarz symmetrization of the level sets $\Omega_t=\{x\,:\,u(x)>t\}\,$, with $u$ a positive eigenfunction associated with $\lambda_1(\Omega)$ and $t>0\,$. We go through the same steps as in \cite[I.9]{BeMe}, or \cite[III.3]{Ch}.  

To apply the symmetrization method in this case, we have to check whether a (geodesic) disk of area less than $A(\Omega_t)$ can be contained on the torus $\mathbb{T}^2$, that is to say whether its radius is no larger than $\frac{1}{2}\,$. Since $A(\Omega_t)\le A(\Omega)\le \frac{1}{\pi}$ and since a disk of area $\frac{1}{\pi}$ has radius $\frac{1}{\pi}<\frac{1}{2}\,$, this is indeed the case. 

It is possible that for some values of $t>0\,$, the set $\Omega_t$ does not satisfy Properties (i) and (ii), but we can avoid any potential problem thanks to the approximation argument presented in \cite[I.9]{BeMe}.  In this paper, the authors use the genericity  results of \cite{U} to show that an arbitrarily small zero-order perturbation of $-\Delta$ gives an operator whose first eigenfunctions are Morse functions in $\Omega$ with a finite number of critical points.  

\section{Proof of the main theorem}
	\label{secProof} 

\subsection{Weyl's law with explicit bounds}

Let us now go back to the study of the sequence $(\lambda_k(\mathbb{T}^2))_{k\ge 1}\,$. For $\lambda\ge 0\,$, we write 
\[N(\lambda)=\sharp\{k\,:\,\lambda_k(\mathbb{T}^2)\le \lambda\}\]
(this is the \emph{counting function}).  According to Weyl's law, which holds for any compact Riemannian manifold, we have
\begin{equation*}
N(\lambda)\sim \frac{1}{(2\pi)^2}\omega_2\, A(\mathbb{T}^2)\lambda \mbox{ as }  \lambda \to +\infty \,,
\end{equation*}
where $\omega_2$ is the area of a disk of radius $1$. We obtain
\begin{equation}
\label{eqWeylsLaw}
N(\lambda)\sim \frac{\lambda}{4\pi} \mbox{ as }  \lambda \to +\infty \,.
\end{equation}
\\
To reach our goal, we will need to replace Formula \eqref{eqWeylsLaw} with an explicit lower bound for $N(\lambda)$.

\begin{prop} We have, for all $\lambda\ge 0\,$,
\begin{equation}
\label{eqWeylExplicit}
\frac{\lambda}{4\pi}-\frac{2\sqrt{\lambda}}{\pi}-3\le N(\lambda)\,.
\end{equation}
\end{prop}

\begin{proof}
	The eigenvalues of $-\Delta_{\mathbb{T}^2}$ are of  the form 
	\[\lambda_{m,n}=4\pi^2(m^2+n^2)\,,\]
	with $(m,n)\in \mathbb{N}^2\,$. \\
	With each pair  of integers $(m,n)$ we associate a finite dimensional space $E_{m,n}$ of eigenfunctions such that 
	\[L^2(\mathbb{T}^2)=\overline{\bigoplus_{(m,n)}E_{m,n}}\,.\]  
	
The vector space $E_{m,n}$ is generated by products of trigonometric functions, see for instance the proof of \cite[Theorem 2.2]{HH} for details. The dimension of $E_{m,n}$  equals
	\begin{itemize}
		\item $1$ if $(m,n)=(0,0)$\,;
		\item $2$ if either $m$ or $n$, but not both, is $0$\,;
		\item $4$ if $m>0$ and $n>0$\,.
	\end{itemize}

	For $r\ge 0\,$, let us write $D\left(0,r\right)$ for the closed disk in $\mathbb{R}^2$ of center $0$ and radius $r\,$. We denote by $\mathcal{R}_{\lambda}$ the set
	\[D\left(0,\frac{\sqrt{\lambda}}{2\pi}\right)\cap \left\{(x,y)\in \mathbb{R}^2\,:\, x\ge 0 \mbox{ and } y\ge 0\right\}\,.\] 
	and we write
	\[n(\lambda)=\sharp\left(\mathbb{N}^2\cap\mathcal{R}_{\lambda}\right)\,.\]
	Taking the dimension of the spaces $E_{m,n}$ into account, we have the following exact formula for the counting function:
	\begin{equation}
	\label{eqCountingFunction}
	N(\lambda)=4n(\lambda)-4\left\lfloor \frac{\sqrt{\lambda}}{2\pi}\right\rfloor-3\,.
	\end{equation}
	Let us now obtain a lower bound for $n(\lambda)$. To each point $(m,n)$ in $\mathcal{R}_{\lambda}\,$, we associate the square $[m,m+1]\times[n,n+1]\,$, of area $1$. We define
	\[\mathcal{S}_{\lambda}=\bigcup_{(m,n)\in\left(\mathbb{N}^2\cap\mathcal{R}_{\lambda}\right)}[m,m+1]\times[n,n+1]\,.\]
	The areas of $\mathcal{S}_{\lambda}$ and $\mathcal{R}_{\lambda}$ are  $n(\lambda)$ and $\frac{\lambda}{16\pi}$ respectively. Since $\mathcal{R}_{\lambda}\subset\mathcal{S}_{\lambda}$,
	we have
	\begin{equation}
	\label{eqIneqLattice}
	\frac{\lambda}{16\pi}\le n(\lambda)\,.
	\end{equation}
	Using Inequality \eqref{eqIneqLattice} and Equation \eqref{eqCountingFunction}, we obtain Inequality \eqref{eqWeylExplicit}.
\end{proof}

\subsection{Courant-sharp eigenvalues on the torus}

We now turn to the proof Theorem \ref{thmPlejielTorus}. We will use the following lemmas.

\begin{lem} 
\label{lemPleijel}
If $\lambda$ is an eigenvalue of  $-\Delta_{\mathbb{T}^2}$ that has an associated eigenfunction $u$ with $k$ nodal domains, and if $k\ge 4\,$, 
\[\pi j_{0,1}^2 k\le \lambda \,.\]
\end{lem}
\begin{proof} Since $A(\mathbb{T}^2)=1\,$, one of the nodal domains of $u$ has an area no larger than $\frac{1}{k}\,$. Let us denote this nodal domain by $D\,$. Since $k\ge 4\,$, $A(D)\le \frac{1}{\pi}\,$. According to Proposition \ref{propFaberKrahn},
\[\lambda=\lambda_1(D)\ge \frac{\pi j_{0,1}^2}{A(D)}\ge \pi j_{0,1}^2 k\,.\]
\end{proof}

\begin{cor}
 \label{corCourantSharp}
 If $\lambda$ is a Courant-sharp eigenvalue of $\mathbb{T}^2\,$ with $\nu(\lambda)\ge 4\,$,  
\begin{equation}
\label{eqIneqCourantSharp}
\pi j_{0,1}^2 \nu(\lambda) \le\lambda\,.
\end{equation}
\end{cor}

\begin{lem} 
\label{lemUpperLambda}
For all all $k\in \mathbb{N}\,$,
\begin{equation*}
 	\lambda_k(\mathbb{T}^2)\le \left(4+2\sqrt{4+\pi(k+3)}\right)^2\,.
\end{equation*}
\end{lem}

\begin{proof} The proof is immediate from the following remark: if $\lambda$  is  a non-negative number such that $N(\lambda)\ge k\,$, then $\lambda_k(\mathbb{T}^2)\le \lambda\,$.  The lower bound for $N(\lambda)$ given in Inequality \eqref{eqWeylExplicit} then implies the desired upper bound.
\end{proof}

A direct computation shows that if 
\[k>\frac{\left(4j_{0,1}+2\sqrt{4j_{0,1}^2+3\pi(j_{0,1}^2-4)}\right)^2}{\pi(j_{0,1}^2-4)^2}\simeq 49.5973\,,\]
we have
\[\left(4+2\sqrt{4+\pi(k+3)}\right)^2<\pi j_{0,1}^2 k\,.\]
Lemma \ref{lemUpperLambda} and Corollary \ref{corCourantSharp} then show that if $\lambda$ is an eigenvalue of $-\Delta_{\mathbb{T}^2}$ with $\nu(\lambda)\ge 50\,$, $\lambda$ is not Courant-sharp. 
Table \ref{tabEigVal} gives the first fifty-seven eigenvalues of  $-\Delta_{\mathbb{T}^2}\,$.
\begin{table}
\centering
	\begin{tabular}{|c|c|c|c|}
	\hline
	$\frac{\lambda}{4\pi^2}$&indices&multiplicity&sum of multiplicities\\\hline
	$0$& $(0,0)$ & $1$ & $1$\\
	$1$& $(1,0)$, $(0,1)$ & $4$ & $5$\\
	$2$& $(1,1)$ & $4$ & $9$\\  
	$4$& $(2,0)$, $(0,2)$ & $4$ & $13$\\
	$5$& $(2,1)$, $(1,2)$ & $8$ & $21$\\
	$8$& $(2,2)$ & $4$ & $25$\\
	$9$& $(3,0)$, $(0,3)$ & $4$ & $29$\\
	$10$& $(3,1)$, $(1,3)$ & $8$ & $37$\\
	$13$& $(3,2)$, $(2,3)$ & $8$ & $45$\\
	$16$& $(4,0)$, $(0,4)$ & $4$ & $49$\\
	$17$& $(4,1)$, $(1,4)$ & $8$ & $57$\\
	\hline
	\end{tabular}
	\caption{The first $57$ eigenvalues\label{tabEigVal}}
\end{table}
 In particular, we find 
\[\frac{\lambda_{49}(\mathbb{T}^2)}{4\pi^2}=16\]
and
\[\frac{\lambda_{50}(\mathbb{T}^2)}{4\pi^2}=17\,.\]
Therefore, for all $k \ge 50\,$,  $\nu(\lambda_k(\mathbb{T}^2))\ge 50\,$, and thus $\lambda_k(\mathbb{T}^2)$ is not Courant-sharp.

According to Table \ref{tabEigVal}, it remains to test Inequality \eqref{eqIneqCourantSharp} for $\lambda=\lambda_k(\mathbb{T}^2)$ with $k \in \{6,10,14,22,26,30,38,46\}$ in order to conclude the proof of Theorem \ref{thmPlejielTorus}. We can rewrite \eqref{eqIneqCourantSharp} in the more convenient form 
\[\frac{j_{0,1}^2}{4\pi}\le \frac{\lambda}{4\nu(\lambda)\pi^2}\,.\]
\begin{table}
\centering
	\begin{tabular}{|c|c|c|c|c|c|c|c|c|}
	\hline
	$k$ & $6$ & $10$ & $14$ & $22$ & $26$ & $30$ & $38$ & $46$\\
	\hline
	$\frac{\lambda_{k}(\mathbb{T}^2)}{4k\pi^2}$       & $0.3333$ & $0.4000$ & $0.3571$ & $0.3636$ & $0.3462$ & $0.3333$ & $0.3421$ & $0.3478$\\
	\hline
	\end{tabular}
	\caption{Table of ratios\label{tabRatio}}
\end{table}
Table \ref{tabRatio} shows the ratio 
\[\frac{\lambda_{k}(\mathbb{T}^2)}{4k\pi^2}\]
as a function of $k\,$. Since 
\[\frac{j_{0,1}^2}{4\pi}\simeq 0.4602\,,\]
the inequality is not satisfied for any of the $k$ in consideration. This implies that only the eigenvalues $\lambda_1(\mathbb{T}^2)=0$ and $\lambda_2(\mathbb{T}^2)=\lambda_3(\mathbb{T}^2)=\lambda_4(\mathbb{T}^2)=\lambda_5(\mathbb{T}^2)=4\pi^2$ are Courant-sharp.

{\small

}

\end{document}